\documentclass[12pt, reqno]{amsart}
\usepackage{amssymb,dsfont}
\usepackage{eucal}
\usepackage{amsmath}
\usepackage{amscd}
\usepackage[dvips]{color}
\usepackage{multicol}
\usepackage[all]{xy}           
\usepackage{graphicx}
\usepackage{color}
\usepackage{colordvi}
\usepackage{xspace}
\usepackage{txfonts}
\usepackage{lscape}
\usepackage{tikz} 

\numberwithin{equation}{section}
\usepackage[shortlabels]{enumitem}
\usepackage{ifpdf}
\ifpdf
 \usepackage[colorlinks,final,hyperindex]{hyperref}
\else
\fi
\usepackage{tikz}
\usepackage[active]{srcltx}

\makeatletter

\newtheorem{theorem}{Theorem}[section]
\newtheorem{definition}[theorem]{Definition}
\newtheorem{lemma}[theorem]{Lemma}
\newtheorem{proposition}[theorem]{Proposition}
\newtheorem{remark}[theorem]{Remark}
\newtheorem{example}[theorem]{Example}
\newtheorem{claim}{Claim}
\numberwithin{equation}{section}

\def\al{\alpha}

\def\be{\beta}

\def\ga{\gamma}

\newcommand{\N}{{\mathbb N}}
\newcommand{\C}{{\mathbb C}}

\newcommand{\Z}{{\mathbb Z}}

\def\C{{\mathbb C}}

\def\GG{\mathcal{G}}
\def\UU{\mathcal{U}}
\def\VV{\mathcal{V}}

\def\DD{\mathcal{D}}
\def\HH{\mathcal{H}}
\def\RR{\mathcal{R}}

\def\beq{\begin{eqnarray}}
\def\eeq{\end{eqnarray}}
\def\beqs{\begin{eqnarray*}}
\def\eeqs{\end{eqnarray*}}

\def\l{\lambda}

\begin{document}
\title[Tensor product weight modules]
{Tensor product weight modules for the mirror-twisted Heisenberg-Virasoro algebra}
  \author{Dongfang Gao}
  \address{D. Gao: School of Mathematical Sciences, University of Science and Technology of China,
  Hefei, Anhui 230026, P. R. China}
  \email{gaodfw@mail.ustc.edu.cn}
  \author{Kaiming Zhao}
  \address{K. Zhao: Department of Mathematics, Wilfrid Laurier University,
  Waterloo, ON N2L 3C5, Canada}
  \email{kzhao@wlu.ca}

\maketitle
\begin{abstract} In this paper, we study irreducible weight modules with infinite dimensional weight spaces over the mirror-twisted Heisenberg-Virasoro
algebra $\DD$. More precisely,
the necessary and sufficient conditions for the tensor products of irreducible highest weight modules
and irreducible modules of intermediates series over  $\DD$ to be irreducible are determined by using ``shifting technique".
This leads to a family of new irreducible weight modules over  $\DD$.
Then we   obtain that any two
such tensor products are isomorphic if and only if the corresponding
highest weight modules and modules of intermediate series are
isomorphic respectively.
Also we   discuss submodules of the tensor product module when it is not irreducible.
\end{abstract}

\section{Introduction}\label{intro}
It is well-known that the Virasoro algebra $\VV$ is one of the most important Lie algebras
in mathematics and in mathematical physics, because of its widespread applications, for example,  in  quantum physics (see \cite{GO}),
conformal field theory (see \cite{DMS}), vertex operator algebras (see \cite{DMZ,FZ}), and so on.
Many other interesting and important algebras are closely related to the Virasoro algebra, such as the  Schr{\"o}dinger-Virasoro algebra (see \cite{H,H1}), the Heisenberg-Virasoro algebra (see \cite{ACKP, LZ1}) and the mirror-twisted Heisenberg-Virasoro algebra $\DD$ (see \cite{B, LPXZ}) which is the even part of the mirror-twisted $N=2$ superconformal algebra (see \cite{B}).
The mirror-twisted Heisenberg-Virasoro algebra $\DD$  has a   nice structure (see Definition \ref{def.1})
which is similar to the Heisenberg-Virasoro algebra.
This algebra is the main object that we concern in this paper.

Highest weight modules and Harish-Chandra modules over many Lie algebras with triangular decompositions
have been  the most popular modules in representation theory. These modules are well understood for many infinite dimensional Lie algebras, for example,
the Virasoro algebra in \cite{A,FF,M},  the Heisenberg-Virasoro algebra  in \cite{ACKP, LZ1},
the Schr{\"o}dinger-Virasoro algebra in \cite{LS,TZ},
and the mirror-twisted Heisenberg-Virasoro algebra in \cite{LPXZ}.

For the last two decades more and more mathematicians are  interested in
irreducible weight modules with infinite dimensional weight spaces, see for example, \cite{CP, CM, Z}.
An efficient way to obtain irreducible weight module with infinite dimensional weight spaces is the  tensor product of irreducible highest weight modules
and modules of intermediate series (or loop modules), see for example,  \cite{CP, GZ} for affine Kac-Moody algebras,
\cite{CGZ,Z} for the Virasoro algebra,    \cite{LZ2} for the Heisenberg-Virasoro algebra.
But it is generally not easy to determine the conditions for the tensor product to be irreducible (see \cite{CGZ, LZ, LZ2}).
In the present paper, we will  determine the necessary and sufficient conditions for the tensor product of an
irreducible highest weight module and an irreducible module of intermediate series over $\DD$ to be irreducible.

The paper is organized as follows. In Section 2, we recall    notations related to the mirror-twisted Heisenberg-Virasoro algebra $\DD$
and collect some known results on highest weight modules and module of intermediates series over $\DD$, including also two important lemmas on  tensor product modules for later use.
In Section 3, we determine the necessary and sufficient conditions for the tensor product of an
irreducible highest weight module and  an irreducible module of intermediate series over $\DD$ to be irreducible. Consequently we obtain  a lot of new irreducible weight modules with infinite dimensional weight spaces over  $\DD$, see Theorem \ref{main.15}.
In Section 4, we show that  two such tensor product modules to be isomorphic
if and only if the corresponding highest weight modules and modules of intermediate series are
isomorphic respectively, see Theorem \ref{main.17}.
In Section 5, we apply the results established in Section 4  to construct
irreducible weight $\DD$-modules with infinite dimensional weight spaces.
In these examples, when the tensor product module is not irreducible we also discuss  its submodules.

Throughout this paper, we denote by $\Z,\N,\Z_+, \C$ and $\C^*$
the sets of integers, positive integers, non-negative integers, complex numbers and nonzero complex numbers respectively. All vector spaces and Lie algebras are over $\C$. We denote by $\UU{(\GG)}$ the universal enveloping algebra for a Lie algebra $\GG$.

\section{Notations and preliminaries}\label{pre}
In this section we recall some notations and collect some results for the future convenience.
\begin{definition}\label{def.1}
The {\bf mirror-twisted Heisenberg-Virasoro algebra} $\DD$ is a Lie algebra with basis
$\{d_m, h_r,{\bf{c}},{\bf{l}}|m\in\Z, r\in \frac{1}{2}+\Z\}$ defined by the following commutation relations
\begin{equation*}
\begin{split}
&[d_m, d_n]=(m-n)d_{m+n}+\frac{m^3-m}{12}{\delta}_{m+n, 0}{{\bf{c}}},\\
&[d_m, h_r]=-rh_{m+r},                 \\
&[h_r, h_s]=r{\delta}_{r+s, 0}{\bf{l}},\\
&[{\bf{c}},\DD]=[{\bf{l}},\DD]=0,
\end{split}
\end{equation*}
for $m,n\in\Z, r,s\in \frac{1}{2}+\Z$.
\end{definition}
The Lie subalgebra spanned by $\{d_m, {\bf{c}}| m\in\Z \}$ is the Virasoro algebra $\VV$,
and the Lie subalgebra spanned by $\{h_r, {\bf{l}}|r\in \frac{1}{2}+\Z \}$
is the twisted Heisenberg algebra $\HH$.
Moreover, $\HH$ is an ideal of $\DD$
and $\DD$ is the semi-direct product of $\VV$ and $\HH$.
It is clear that $\DD, \VV, \HH$ are all $\Z$-graded Lie algebras, and their universal enveloping algebras are also $\Z$-graded. The homogeneous spaces of $\UU(\DD)$ are actually the eigenspaces with respect to $-d_0$. We call the eigenvalue of  a  homogeneous element the {\bf degree}, say $\deg(x)=k$ if $x\in \UU(\DD)_k$ for $k\in\frac{\Z}2$.

 Note that $\DD$  has the triangular decomposition:
\begin{equation*}
\DD=\DD^+\oplus\DD^0\oplus\DD^- ,
\end{equation*}
where
\begin{equation*}
\DD^{\pm}=\bigoplus_{n\in\N}\C d_{\pm n}\oplus\bigoplus_{r\in\frac{1}{2}+\Z_+}\C h_{\pm r},
\quad\DD^0=\C d_0\oplus\C{\bf{c}}\oplus\C{\bf{l}}.
\end{equation*}

Similar, $\VV$ and $\HH$ also have the triangular decompositions.
For convenience, we define the following notations for $r,s\in\Z_+,k\in\frac{\Z}{2}, z\in\C$,
\begin{equation*}
\begin{aligned}
&\VV^{(r)}=\bigoplus_{i=r+1}^{\infty}\C d_i+\C{\bf{c}},
\quad\DD^{(r,s)}=\bigoplus_{i=r+1}^{\infty}\C d_i\oplus
\bigoplus_{j=s}^{\infty}\C h_{\frac{1}{2}+j}+\C{\bf{c}}+\C{\bf{l}},\\
&\delta_{z\in \Z}=\begin{cases}&1 \quad\text{ if } z\in \Z\\ &0 \quad\text{ if } z\notin \Z,\end{cases}
 \  \  \,\,\,\,  \delta_{z\notin \Z}=\begin{cases}&0 \quad\text{ if } z\in \Z\\ &1 \quad\text{ if } z\notin \Z.\end{cases}
\end{aligned}
\end{equation*}

\begin{definition}
A $\DD$-module $V$ is called a $\bf{weight}$ $\bf{module}$ if $V=\oplus_{\l\in\C}V_{\l}$,
where $V_{\l}=\{v\in V|d_0v=\l v\}$. The subspace  $V_{\l}$ is called the {\bf weight space} of $V$ corresponding to the weight $\l$.
\end{definition}

\begin{definition}
Let $\GG=\oplus_{i\in\Z}{\GG}_{i}$ be a $\Z$-graded Lie algebra. A $\GG$-module $V$ is called  the $\bf{restricted}$ module if for any $v\in V$
there exists $n\in\N$ such that ${\GG}_iv=0$, for $i>n$.
The  category of restricted modules over  $\GG$ will be denoted as  $\RR_{\GG}$.
\end{definition}

Note that if $V$ is a $\VV$-module,
then $V$ can be easily viewed as a $\DD$-module by defining $\HH V=0$,
the resulting  module is denoted by $V^{\DD}$.
Thanks to \cite{FLM}, for any $H\in\RR_{\HH}$ with the action of $\bf l$ as a nonzero scalar $l$, we can give $H$ a $\DD$-module structure
denoted by $H^{{\DD}}$ via the following map
\begin{eqnarray}\label{Vertex}
&&d_n\mapsto \frac{1}{2l}\sum_{k\in\Z+\frac{1}{2}}h_{n-k}h_k,\quad\forall n\in\Z, n\neq0,\\
&&d_0\mapsto \frac{1}{2l}\sum_{k\in\Z+\frac{1}{2}}h_{-|k|}h_{|k|}+\frac{1}{16},\\
&&h_r\mapsto h_r,\quad\forall r\in\frac{1}{2}+\Z,
\quad {\bf{c}} \mapsto 1,
\quad {\bf{l}}\mapsto l.
\end{eqnarray}

\subsection{Modules of intermediate series}
\begin{definition}
A weight $\DD$-module $V$ is called a module of $\bf{intermediate}$ $\bf{series}$
if all its nonzero weight spaces are one-dimensional.
\end{definition}
Recall that for all $\al,\be\in\C \text{ and } \ga\in\C^*$,
$A(\al,\be,\ga)=\oplus_{k\in\frac{1}{2}\Z}\C v_k$ is an irreducible $\DD$-module with the actions defined by
\begin{equation}\label{A(a,b,g)}
\begin{split}
&d_mv_k=(\al+\be m-k)v_{m+k},\\
&h_rv_n=v_{n+r},\\
&h_rv_s=\ga v_{r+s},\\
&{\bf{c}}v_k={\bf{l}}v_k=0,
\end{split}
\end{equation}
for ${m,n}\in\Z$, $r,s\in \frac{1}{2}+\Z$, $k\in\frac{1}{2}\Z$.

Let us recall the  $\VV$-module of the intermediate series $A(\al,\be)=\oplus_{k\in\Z}\C v_k$:
\begin{equation}\label{A(a,b)}d_mv_k=(\al+\be m-k)v_{m+k},\,\,\,{\bf{c}}v_k=0,
\end{equation}
for ${m,k}\in\Z$.
We denote by $A'(\al,\be)$ the unique nontrivial irreducible subquotient module of $A(\al,\be)$.
It is clear that $A(\al,\be)^{\DD}$ is an irreducible $\DD$-module
if and only if $A(\al,\be)$ is an irreducible $\VV$-module.
Moreover, we have the following well-known results:
\begin{lemma}\label{main.1}
\begin{enumerate}[$(1)$]
\item The $\VV$-module $A(\al,\be)$ is isomorphic to $\VV$-module $A(\al+m,\be)$ for $m\in\Z$.
\item The $\VV$-module $A(\al,\be)$ is irreducible if and only if
$\al\notin\Z$ or $\be\not=0,-1$.
\item If $\al\in\Z$, $\be=0$, then $A(\al,\be)$ has a unique nontrivial irreducible quotient module.
\item If $\al\in\Z$, $\be=-1$, then $A(\al,\be)$ has a unique nontrivial irreducible submodule.
\item  $A'(\al,0)\cong A'(\al,-1)$ for $\al\in\C$.
\end{enumerate}
\end{lemma}

We have the classification for all  irreducible modules of intermediate series over $\DD$.

\begin{lemma}(cf. \cite[ Theorem 3.9]{LPXZ})\label{main.2}
If $V$ is an irreducible module of intermediate series over $\DD$,
then $V\cong A(\al,\be,\ga)$ or $V\cong A'(\al,\be)^{\DD}$
for some $\al,\be\in\C$ and $\ga\in\C^*$.
\end{lemma}
\subsection{Verma modules}
\begin{definition}
A  weight $\DD$-module $V$ is called a
{\bf highest weight module} with {highest weight} $\l\in \C$,
if there exists a nonzero weight vector $v
		\in V_{\l}$ such that $V$ is generated by $v$ as  $\DD$-module and $\DD^+ v=0 $.
\end{definition}

\begin{definition}
For $(c,h,l)\in\C^3$, let $\C{\bf1}$ be a $(\DD^0\oplus\DD^+)$-module with
		$$
		d_0{\bf1}=h{\bf1},\quad {\bf{c}{\bf1}}=c{\bf1}, \quad \mathbf {l}\cdot{{\bf1}}=l{\bf1},\quad \DD^{+}{\bf1}=0.
		$$
		The {\bf Verma module} over ${\DD}$ is defined by
		$$
		M(c,h,l)=U(\DD)\otimes_{U(\DD^{0}\oplus\DD^{+})}\C{\bf1}.
		$$
\end{definition}
Let $J(c,h,l)$ be the unique maximal submodule of $M(c,h,l)$.
Then we have the irreducible highest weight module $L(c,h,l)=M(c,h,l)/J(c,h,l)$.
We will use the ${\bf1}$ and ${\overline{\bf1}}$ to denote the highest weight vectors
of $M(c,h,l)$ and $L(c,h,l)$ respectively.
Similarly, we have $M_{\VV}(c,h),L_{\VV}(c,h)=M_{\VV}(c,h)/J_{\VV}(c,h)$ for the Virasoro algebra $\VV$,
and $M_{\HH}(l),L_{\HH}(l)=M_{\HH}(l)/J_{\HH}(l)$ for the twisted Heisenberg algebra $\HH$.
It is clear that they are all restricted modules.
Moreover, we have the following results about highest weight modules over $\DD$.
\begin{lemma}(cf. \cite[Proposition 5.1]{LPXZ})\label{main.3}
Let $c,h\in\C$ and $l\in\C^*$. Then we have $\DD$-module isomorphisms
\begin{eqnarray*}
&&\pi_1:M(c,h,l)\rightarrow M_{{\VV}}(c-1,h-\frac{1}{16})^{\DD}\otimes M_{\HH}(l)^{\DD},\\
&&\pi_2:L(c,h,l)\rightarrow L_{{\VV}}(c-1,h-\frac{1}{16})^{\DD}\otimes M_{\HH}(l)^{\DD}.
\end{eqnarray*}
\end{lemma}


\begin{lemma}(cf. \cite[Corollary 5.2]{LPXZ})\label{main.4}
Let $c,h\in\C$ and $l\in\C^*$.
Then the Verma module $M(c,h,l)$  is irreducible if and only if
$\phi_{r,s}(c-1,h-\frac{1}{16})\neq 0$ for all $r,s\in\Z_+$, where
\begin{eqnarray*}
			\phi_{r,s}(c,h)&=&h-\frac{1}{48}(13-c)(r^2+s^2)\\
			&&-\frac{1}{48}\left(\sqrt{(c-1)(c-25)}(r^2-s^2)-24rs-2+2c\right).
		\end{eqnarray*}
\end{lemma}
We conclude this section by recalling two familiar results for later use.
\begin{lemma}(cf. \cite[Lemma 6]{LZ})\label{main.5}
Let $\GG$ be a Lie algebra over $\C$ with a countable basis, and   $V$ be an  irreducible $\GG$-module.
For any $n\in \N$ and any linearly independent subset
$\{v_1,v_2,\cdots v_n\}\subset V$, and any subset
$\{v_1',\ldots,v_n'\}\subset V$, there exists some $u\in \UU{(\GG)}$, such
that
 $$uv_i=v_i',\forall i=1,2,\cdots, n.$$
 \end{lemma}

\begin{lemma}(cf. \cite[Theorem 7]{LZ})\label{main.6}
Let $\GG$ be a Lie algebra over $\C$ with a  countable basis, and  $V_1, V_2$ be $\GG$-modules.
Suppose that one of the following conditions holds:

\begin{enumerate}[$(1)$]
\item The module $V_1$ is irreducible and $Ann_{\GG}(v)+Ann_{\GG}(S)=\GG$ for all $v\in V_1$
and all finite subsets $S\subseteq V_2$;
\item
 For any finite subset
$S\subseteq V_2$, $V_1$ is an irreducible
$Ann_{\GG}(S)$-module.\end{enumerate} Then the following statements hold.
\begin{enumerate}[$($a$)$]
 \item  Any
submodule of $V_1\otimes V_2$ is of the form $V_1\otimes V_2'$ for
a submodule $V_2'$  of $V_2$; \item If $V_1, V_2$ are irreducible, then
$V_1\otimes V_2$ is irreducible.
\end{enumerate}
\end{lemma}

\section{Irreducibility of tensor product weight modules}
From Section 2, we know that all irreducible modules of intermediate series over $\DD$
are $A'(\al,\be)^{\DD}$ or $A(\al,\be,\ga)$ where  $\al,\be\in\C,\ga\in\C^*$.
In this section, we will determine the necessary and sufficient conditions
for the tensor product modules $L(c,h,l)\otimes A'(\al,\be)^{\DD},L(c,h,l)\otimes A(\al,\be,\ga)$ to be irreducible
for any $c,h,l,\al,\be\in\C,\ga\in\C^*$.
\begin{lemma}\label{main.7}
Suppose that $V, W$ are  $\VV$-modules, and $H,K\in \RR_{\HH}$ are irreducible with nonzero action of $\bf{l}$.
Then
 \begin{enumerate}[$(1)$]
 \item  any  $\DD$-submodule of  $V^{\DD}\otimes H^{\DD}$ is of the
form $(V')^{\DD}\otimes  H^{\DD}$ for some $\VV$-submodule $V'$ of
$V$. In particular, $V^{\DD}\otimes H^{\DD}$ is an irreducible $\DD$-module
if and only if $V$ is an irreducible $\VV$-module;
 \item $V^{\DD}\otimes H^{\DD}\cong W^{\DD}\otimes K^{\DD}$ if and
only if $V\cong W$ and $H\cong K$.
 \end{enumerate}
\end{lemma}
\begin{proof}(1) Note that $\HH\subseteq Ann_{\DD}({V^{\DD}})$, and $H^{\DD}$ is  an irreducible $\HH$-module.
Thus the result follows from Lemma \ref{main.6}.

 (2) It is clear for the sufficiency.

 Now we show the necessity. Suppose that $\varphi: V^{\DD}\otimes H^{\DD}\rightarrow W^{\DD}\otimes K^{\DD}$ is a $\DD$-module isomorphism.
 Let's fix a nonzero $x_0\in H$.
 For any $0\ne v\in\VV$,
 we write $\varphi(v\otimes x_0)=\sum_{i=1}^{t}w_i\otimes y_i$ with minimal $t$.
 We know that $y_i$'s are linearly independent.
 Then there exists some $u_0\in \UU{(\HH)}$ such that $u_0y_i=\delta_{1,i}y_i,i=1,2,\cdots,t$ by Lemma \ref{main.5}.
 Therefore,
 $$\varphi(v\otimes uu_0x_0)=\varphi(uu_0(v\otimes x_0))=w_1\otimes uy_1,\forall u\in \UU{(\HH)}.$$
 Denote $x_1=u_0x_0$. Then we have $\varphi(v\otimes ux_1)=w_1\otimes uy_1,\forall u\in \UU{(\HH)}.$
 Notice $\UU{(\HH)}x_1=H, \UU{(\HH)}y_1=K$, since $H,K$ are irreducible.
Define $\varphi_v:H\rightarrow K$ by $\varphi_v(ux_1)=uy_1,$ for all $u\in \UU{(\HH)}$.
 It is easy to see that $\varphi_v$ is a $\UU{(\HH)}$-module isomorphism.
That shows $H\cong K$ as $\HH$-modules.

At the moment, we can assume that $H=K$.
Then $\varphi$ is an isomorphism
from $V^{\DD}\otimes H^{\DD}$ to $W^{\DD}\otimes H^{\DD}$,
and $\varphi_v$ is an automorphism of $H$.
In fact, $\varphi_v$ is a scalar denoted by $c_v$,
since $H$ is irreducible and $\UU{(\HH)}$ is countably generated.
Then for any $v\in V,x\in H$,
$$\varphi(v\otimes x)=c_vw_v\otimes x,$$
for some $w_v\in H$.
Define the map $\varpi:V\rightarrow W$ by $v\rightarrow c_vw_v$.
It's clear that $\varpi$ is bijective.
Furthermore, we have that $d_n\varpi(v)\otimes x=\varpi(d_n v)\otimes x$,
since $d_n\varphi(v\otimes x)=\varphi (d_n(v\otimes x)),$ for all $v\in V,x\in H$.
Therefore, $d_n\varpi(v)=\varpi(d_n v)$,
that means $\varpi$ is a $\VV$-module isomorphism.
So, $V\cong W$ as $\VV$-modules.
We complete the proof.
   \end{proof}

\begin{proposition}\label{main.8}
Let $c,h,l,\al,\be\in\C,\ga\in\C^*$.
\begin{enumerate}[$(1)$]
  \item Suppose $l\ne0$. Then $L(c,h,l)\otimes A'(\al,\be)^{\DD}$ is an irreducible $\DD$-module
  if and only if $L_{\VV}(c-1,h-\frac{1}{16})\otimes A'(\al,\be)$ is an irreducible $\VV$-module.
  \item Suppose $l=0$. Then $L(c,h,l)\otimes A'(\al,\be)^{\DD}$ is an irreducible $\DD$-module
  if and only if $L_{\VV}(c,h)\otimes A'(\al,\be)$ is an irreducible $\VV$-module.
  \item Suppose $l=0$, then $L(c,h,l)\otimes A(\al,\be,\ga)$ is irreducible.
\end{enumerate}
\end{proposition}

\begin{proof}
(1) From Lemma \ref{main.3} we know that
$$L(c,h,l)\cong L_{\VV}(c-1,h-\frac{1}{16})^{\DD}\otimes M_{\HH}(l)^{\DD}.$$
So we have
$$L(c,h,l)\otimes A'(\al,\be)^{\DD}\cong M_{\HH}(l)^{\DD}\otimes (L_{\VV}(c-1,h-\frac{1}{16})\otimes A'(\al,\be))^{\DD}.$$
It is clear that $M_{\HH}(l)\in\RR_{\HH}$ is irreducible when $l\neq 0$.
Therefore, the result follows from Lemma \ref{main.7}.

(2)  Since $l=0$, we see that $L(c,h,l)\cong L_{\VV}(c,h)^{\DD}$ as $\DD$-modules.
Thus
$$L(c,h,l)\otimes A'(\al,\be)^{\DD}\cong L_{\VV}(c,h)^{\DD}\otimes A'(\al,\be)^{\DD}
\cong (L_{\VV}(c,h)\otimes A'(\al,\be))^{\DD}.$$
Then the result is clear.

(3) Since $l=0$, we  have $L(c,h,l)\cong L_{\VV}(c,h)^{\DD}$ as $\DD$-modules.
Notice that for any finite subset $S\subseteq L_{\VV}(c,h)^\DD$,
there exists some $r$ such that $\VV^{(r)}+\HH\subseteq Ann_{\DD}(S)$.
Moreover, we   see that the   module $A(\al,\be,\ga)$ defined in (\ref{A(a,b,g)}) is irreducible as $(\VV^{(r)}+\HH)$-module by the following two observations:
\begin{enumerate}[$(1)$]
  \item[(a)]  Each $v_k$ can generate $A(\al,\be,\ga)$ as  an $\HH$-submodule;
  \item[(b)]  $h_{\frac{1}{2}}h_{-m-\frac{1}{2}}d_mv_k=\gamma(\alpha-k+m\beta)v_k$  for any $m>r$.
\end{enumerate}
So the result follows from Lemma \ref{main.6}.
\end{proof}

\begin{remark}
In the above theorem, we reduce the irreducibility of $\DD$-module  $L(c,h,l)\otimes A'(\al,\be)^{\DD}$
to the irreducibility of the $\VV$-module $L_{\VV}(c,h)\otimes A'(\al,\be)$
or $L_{\VV}(c-1,h-\frac{1}{16})\otimes A'(\al,\be)$.
Thanks to \cite{CGZ},
the irreducibility of  the $\VV$-module $L_{\VV}(c,h)\otimes A'(\al,\be)$ has been completely determined for $c,h,\al,\be\in\C$.
\end{remark}

Now we have only the most difficult case left:   $L(c,h,l)\otimes A(\al,\be,\ga)$ with $l\ne 0$. To determine the necessary and sufficient conditions for this tensor product module to be  irreducible we will break the arguments into several lemmas.
From now on we assume that $c,h,\al,\be\in\C,l, \ga\in\C^*$, and we will use the definition for $A(\al,\be,\ga)$ in (\ref{A(a,b,g)}).

\begin{lemma}\label{main.9}
(a).
For any $k\in\frac{\Z}{2}$,
 $${\bf{1}}\otimes v_{k}\notin U(\DD)({\bf{1}}\otimes
v_{k+\frac{1}{2}})\subseteq M(c,h,l)\otimes A(\al,\be,\ga).$$

(b). The $\DD$-module $M(c,h,l)\otimes A(\al,\be,\ga)$ is not irreducible.
\end{lemma}
\begin{proof} (a).
Note that for any $k\in\frac{\Z}{2}$,
$$U(\DD)({\bf{1}}\otimes v_{k+\frac{1}{2}})=U(\DD^-)U(\DD^+)({\bf{1}}\otimes
v_{k+\frac{1}{2}})\subseteq \sum_{i\in \Z_+}U(\DD^-)({\bf{1}}\otimes v_{k+\frac{1}{2}+\frac{i}{2}}).$$
So, it is sufficient to show that ${\bf{1}}\otimes v_k\notin \displaystyle{\sum_{i\in \Z_+}}U(\DD^-)({\bf{1}}\otimes v_{k+\frac{1}{2}+\frac{i}{2}}).$

Assume that ${\bf{1}}\otimes v_k\in \displaystyle{\sum_{i\in \Z_+}}U(\DD^-)({\bf{1}}\otimes v_{k+\frac{1}{2}+\frac{i}{2}})$.
Then by PBW Theorem we know that
$${\bf{1}}\otimes v_k=\sum_{i=1}^r u_{-\frac{i}{2}}({\bf{1}}\otimes v_{k+\frac{i}{2}}),$$
where $u_{-\frac{i}{2}}\in U(\DD^-)_{-\frac{i}{2}}$ with $u_{-\frac{r}{2}}\ne0$.
Thus we have $$\sum_{i=1}^r u_{-\frac{i}{2}}({\bf{1}}\otimes v_{k+\frac{i}{2}})-{\bf{1}}\otimes v_k=0,$$
i.e.,
$$\sum_{i=1}^r (u_{-\frac{i}{2}}{\bf{1}}\otimes v_{k+\frac{i}{2}})
+\sum_{i=1}^r ({\bf{1}}\otimes (u_{-\frac{i}{2}}v_{k+\frac{i}{2}})-{\bf{1}}\otimes v_k=0.$$
This is  impossible since the nonzero vector
$u_{-\frac{r}{2}}{\bf{1}}\otimes v_{k+\frac{r}{2}}$ cannot be in the span of the other vectors in the above formula.

(b) follows directly from (a).
\end{proof}

\begin{lemma}\label{main.10} Let $M$ be a nonzero $\DD$-submodule of $L(c,h,l)\otimes A(\al,\be,\ga)$.

(a).  There exists $k\in\frac{\Z}{2}$ such that $\bar{\bf{1}}\otimes
v_{k}\in M$.

(b). There exists $k\in\frac{\Z}{2}$ such that $\bar{\bf{1}}\otimes
v_{i}\in M$ for all $i\in \frac12\Z$ with $i\ge k$.
\end{lemma}
\begin{proof} (a). Since $M$ is a   weight $\DD$-module,  we can take a nonzero weight vector
$$w=\sum_{i=0}^sw_{-\frac{i}{2}}\otimes v_{k+\frac{i}{2}}\in M, \text{ with }w_{0}=\bar{\bf{1}},k\in\frac{\Z}{2}, w_{-\frac{i}{2}}\in U(\DD^-)_{-\frac{i}{2}}\bar{\bf{1}}.$$
There is an  $n\in\N$ such that
$d_jw_{-\frac{i}{2}}=h_{\frac{1}{2}+j}w_{-\frac{i}{2}}=0$ for all $j\ge n, i=0,1,\cdots,s$.
Since $A(\al,\be,\ga)$ is an irreducible $(\VV^{(n)}+\HH)$-module,
by Lemma \ref{main.5}  we may take some $u\in U(\VV^{(n)}+\HH)$
such that $$uv_{k+\frac{i}{2}}=\delta_{0,\frac{i}{2}}v_{0},\forall i=0,1,\ldots,s.$$
We can write $u=\sum_i u_iu_i'$ with $u_i\in U(\HH)$ and $u_i'\in U(\VV^{(n)})$ by PBW Theorem.
For any $v\in\ A(\al,\be,\ga)$, we have
\begin{equation*}
\begin{split}
\begin{cases}
&h_jh_iv=h_{\frac{i+j}{2}}h_{\frac{i+j}{2}} v,\quad {\rm if } \quad i+j\in2\Z+1, \\
&h_jh_iv=h_{\frac{i+j}{2}-\frac{1}{2}}h_{\frac{i+j}{2}+\frac{1}{2}} v,\quad {\rm if } \quad i+j\in2\Z.
\end{cases}
\end{split}
\end{equation*}

Take a sufficiently large $r$. By repeatedly and properly replacing
$h_jh_i$ with $h_{\frac{i+j}{2}}h_{\frac{i+j}{2}}$
or $h_{\frac{i+j}{2}-\frac{1}{2}}h_{\frac{i+j}{2}+\frac{1}{2}}$
in the expression of  $h_ru$ we can result in an element   $u'\in U(\DD^{(n,n)})$ such that
$$u'v_{k+\frac{i}{2}}=h_ruv_{k+\frac{i}{2}}=\delta_{0,\frac{i}{2}}v_{r},\forall i=0,1,\ldots,s.$$
Since $u'w_{-\frac{i}{2}}=0$ for all $i=0,1,\ldots,s$, we deduce that  $\bar{\bf{1}}\otimes v_{r}=u'w\in M$.

(b). This follows by applying $h_i$ to  $\bar{\bf{1}}\otimes
v_{k}\in M$ in (a) for all $i\in \frac{\N}{2}$.
\end{proof}

Now, we need to introduce the ``shifting technique".
For any highest weight $\DD$-module $V$ with highest weight vector ${\bf v}$,
it is easy to check that $V\otimes \C[y^{\pm\frac{1}{2}}]$
can be endowed a $\DD$-module structure by following actions
\begin{equation}\label{eq1}
\begin{split}
&d_n(P{\bf{v}}\otimes y^i)=(d_n+\al+\be n+\deg(P)-i)P{\bf{v}}\otimes y^{i+n},\\
&h_r(P{\bf{v}}\otimes y^i)=(h_r+\delta_{i-\deg(P)\in\Z}+\ga \delta_{i-\deg(P)\notin\Z})P{\bf{v}}\otimes y^{i+r},\\
&{\bf{c}}(P{\bf{v}}\otimes y^i)=c(P{\bf{v}}\otimes y^i),\quad {\bf{l}}(P{\bf{v}}\otimes y^i)=l(P{\bf{v}}\otimes y^i),
\end{split}
\end{equation}
for $n\in\Z,r\in\frac{1}{2}+\Z,i\in\frac{\Z}{2}$ and homogeneous element $P\in \UU(\DD^{-})$.
Then it is not hard to see that
$$V\otimes A(\al,\be,\ga)\cong V\otimes \C[y^{\pm\frac{1}{2}}]$$
as $\DD$-modules.
Specially, we have
\begin{eqnarray*}
&&M(c,h,l)\otimes A(\al,\be,\ga)\cong M(c,h,l)\otimes \C[y^{\pm\frac{1}{2}}],\\
&&L(c,h,l)\otimes A(\al,\be,\ga)\cong L(c,h,l)\otimes \C[y^{\pm\frac{1}{2}}].
\end{eqnarray*}
Henceforth we may view $M(c,h,l)\otimes A(\al,\be,\ga),L(c,h,l)\otimes A(\al,\be,\ga)$ as
$M(c,h,l)\otimes \C[y^{\pm\frac{1}{2}}],L(c,h,l)\otimes \C[y^{\pm\frac{1}{2}}]$ respectively.
Moreover, for any $k\in\frac{\Z}{2}$,
$$M(c,h,l)\otimes y^k=\{v\in M(c,h,l)\otimes \C[y^{\pm\frac{1}{2}}]\mid d_0v=(\al+h-k)v\},$$
thus $ M(c,h,l)\otimes \C[y^{\pm\frac{1}{2}}]=\oplus_{i\in\frac{\Z}{2}}(M(c,h,l)\otimes y^i)$.
Denote
\begin{equation*}
\begin{split}
&W^{(k)}=\sum_{i\in\frac{\Z_+}{2}}\UU(\DD)({\bf{1}}\otimes y^{k+i})\subseteq M(c,h,l)\otimes \C[y^{\pm\frac{1}{2}}],\\
&W^{(k)}_n=W^{(k)}\cap(M(c,h,l)\otimes y^n), \forall n\in\frac{\Z}{2}.
\end{split}
\end{equation*}
We have following observations:

\begin{lemma}\label{main.11}
\begin{enumerate}[$(1)$]
\item $W^{(k)}=\sum_{i\in \frac{\Z_+}{2}} U(\DD^-)({\bf{1}}\otimes y^{k+i})$.
\item $W^{(k)}\supseteq \oplus_{k\leq i\in\frac{\Z}{2}}M(c,h,l)\otimes y^{i}$.
\item $M(c,h,l)\otimes y^{k-\frac{1}{2}}=W^{(k)}_{k-\frac{1}{2}}\oplus \C({\bf{1}}\otimes y^{k-\frac{1}{2}})$.
\item  If $P\in U(\DD^-)$ is homogeneous such that
$$P{\bf{1}}\otimes y^{k-\frac{1}{2}}\in W^{(k)}_{k-\frac{1}{2}},$$
then $(U(\DD^-)P {\bf{1}})\otimes y^{k-\frac{1}{2}} \subseteq W^{(k)}_{k-\frac{1}{2}}$.
\end{enumerate}
\end{lemma}
\begin{proof}
(1) This follows from the following computations:
\begin{equation*}
\begin{split}
W^{(k)}&=\sum_{i\in\frac{\Z_+}{2}}\UU(\DD)({\bf{1}}\otimes y^{k+i})
=\sum_{i\in\frac{\Z_+}{2}}\UU(\DD^-)\UU(\DD^+)({\bf{1}}\otimes y^{k+i}) \\
&=\sum_{i\in\frac{\Z_+}{2}}\UU(\DD^-)(\sum_{j\in\frac{\Z_+}{2}}({\bf{1}}\otimes y^{k+i+j}))
=\sum_{i\in\frac{\Z_+}{2}}\UU(\DD^-)({\bf{1}}\otimes y^{k+i}).
\end{split}
\end{equation*}
(2) By PBW Theorem, each nonzero element of $M(c,h,l)\otimes y^{i}$ can be written as linear combinations of vectors
$$v_{s,t}=h_{-p_s}\dots h_{-p_1}d_{-q_t}\dots d_{-q_1}{\bf{1}}\otimes y^i,$$
where $q_j\in\N,p_{j'}\in \frac{1}{2}+\Z_+$ for $1\leq j\leq t,1\leq j'\leq s$.

Now, we show the result by induction on $s+t$.
If $s+t=0$, then $v_{s,t}={\bf{1}}\otimes y^i\in W^{(k)},\forall k\leq i\in\frac{\Z}{2}$ is trivial.
Suppose that the result is right when $s+t=m$.
For $s+t=m+1$, we may assume that
$$ v_{s,t}=h_{-r}Q{\bf{1}}\otimes y^i \text{ where } Q\in U(\DD^-) ,r>0, \text{ or }$$
$$v_{s,t}=d_{-n}Q{\bf{1}}\otimes y^i, \text{ where } Q\in U(\VV^-) , n>0. $$
Then by equations \ref{eq1} and inductive hypothesis on $Q$, we have
$$\aligned
v_{s,t}=&h_{-r}Q{\bf{1}}\otimes y^i\\
=&h_{-r}(Q{\bf{1}}\otimes y^{i+r})-(\delta_{i+r-\deg(Q)\in\Z}+\ga \delta_{i+r-\deg(Q)\notin\Z})Q{\bf{1}}\otimes y^{i}\in W^{(k)},\text{ or }\\
v_{s,t}=&d_{-n}Q{\bf{1}}\otimes y^i\\
=&d_{-n}(Q{\bf{1}}\otimes y^{i+n})-(\al-\be n+\deg(Q)-i-n)Q{\bf{1}}\otimes y^i\in W^{(k)}.\endaligned $$
Thus $M(c,h,l)\otimes y^{i}\in W^{(k)}, \forall k\leq i\in\frac{\Z}{2}$.

(3) $W^{(k)}_{k-\frac{1}{2}}\oplus \C({\bf{1}}\otimes y^{k-\frac{1}{2}})\subseteq M(c,h,l)\otimes y^{k-\frac{1}{2}}$ is clear by definition.
$W^{(k)}_{k-\frac{1}{2}}\cap\C({\bf{1}}\otimes y^{k-\frac{1}{2}})=0$
follows from the Lemma \ref{main.9}.

Now we show that
$M(c,h,l)\otimes y^{k-\frac{1}{2}}\subseteq W^{(k)}_{k-\frac{1}{2}}\oplus
\C({\bf{1}}\otimes y^{k-\frac{1}{2}})$.
Note that every nonzero element of $M(c,h,l)\otimes y^{k-\frac{1}{2}}$
can be written as linear combinations of vectors
$$v_{s,t}=h_{-p_s}\dots h_{-p_1}d_{-q_t}\dots d_{-q_1}{\bf{1}}\otimes y^{k-\frac{1}{2}},$$
where $q_j\in\N,p_{j'}\in \frac{1}{2}+\Z_+$ for $1\leq j\leq t,1\leq j'\leq s$
by PBW Theorem.
By induction on $s+t$ which is similar to (2), we obtain $v_{s,t}\in W^{(k)}\oplus
\C({\bf{1}}\otimes y^{k-\frac{1}{2}})$.
Thus $M(c,h,l)\otimes y^{k-\frac{1}{2}}\subseteq W^{(k)}\oplus
\C({\bf{1}}\otimes y^{k-\frac{1}{2}})$.
Then $M(c,h,l)\otimes y^{k-\frac{1}{2}}\subseteq W^{(k)}_{k-\frac{1}{2}}\oplus
\C({\bf{1}}\otimes y^{k-\frac{1}{2}})$.

(4) By PBW Theorem, every nonzero element of $(U(\DD^-)P {\bf{1}})\otimes y^{k-\frac{1}{2}}$
can be written as linear combinations of vectors
$$v_{s,t}=h_{-p_s}\dots h_{-p_1}d_{-q_t}\dots d_{-q_1}P{\bf{1}}\otimes y^{k-\frac{1}{2}},$$
where $q_j\in\N,p_{j'}\in \frac{1}{2}+\Z_+$ for $1\leq j\leq t,1\leq j'\leq s$.
By induction on $s+t$ which is similar to (2), we obtain $v_{s,t}\in W^{(k)}$.
Thus $(U(\DD^-)P {\bf{1}})\otimes y^{k-\frac{1}{2}}\in W^{(k)}_{k-\frac{1}{2}}$.
\end{proof}
Now, similar to the definition of $\varphi_{n}$ in \cite{CGZ},
we want to define the linear map $\rho_n:\UU(\DD^-)\rightarrow \C$, for any $n\in \frac{\Z}{2}$.
Let $T(\DD^-)$ denote the tensor algebra of $\DD^-$, i.e.,
$$T(\DD^-)=\C\oplus (\DD^-)\oplus (\DD^-\otimes\DD^-)\oplus\cdots,$$
which is a free associative algebra.
For any $k\in\N$, define $\deg(d_{-k})=-k,\deg(h_{\frac{1}{2}-k})=\frac{1}{2}-k$ and $\deg(1)=0$,
then $T(\DD^-)$ is a $\frac{\Z}{2}$-graded algebra.
For any $n\in \frac{\Z}{2},\al,\be\in\C$ and $\ga\in\C^*$,
we may inductively define the linear map $\overline{\rho_n}:T(\DD^-)\rightarrow \C$ as follows:
\begin{equation}\label{eq2}
\begin{aligned}
&\overline{\rho_n}(1)=1,\\
&\overline{\rho_n}(h_{-r}u)=-(\delta_{n+r+k\in\Z}+\ga \delta_{n+r+k\notin\Z})\overline{\rho_n}(u),\\
&\overline{\rho_n}(d_{-i}u)=-(\al-\be i-k-i-n)\overline{\rho_n}(u),
\end{aligned}
\end{equation}
where $r, i>0$ and $u$ is an arbitrary homogeneous element of degree $-k$ in $T(\DD^-)$.
Note that $\UU(\DD^-)$ is the quotient algebra $T(\DD^-)/I$
where $I$ is the two-sided ideal generated by
\begin{eqnarray*}
&&d_{-i}d_{-j}-d_{-j}d_{-i}-(j-i)d_{-i-j},\\
&&d_{-i}h_{-r}-h_{-r}d_{-i}-rh_{-i-r},\\
&&h_{-r}h_{-r'}-h_{-r'}h_{-r},
\end{eqnarray*}
for $i,j\in\N,r,r'\in \frac{1}{2}+\Z_+$.
It is not hard to see that $I\subseteq ker(\overline{\rho_n})$.
So $\overline{\rho_n}$ induces a linear map $\UU(\DD^-)\rightarrow \C$, denoted by $\rho_n$.
Now we have the following lemma.
\begin{lemma} \label{main.12} Let $P \in \UU(\DD^-),n\in\frac{\Z}{2}$. Then
\begin{enumerate}[$(1)$]
\item $P{\bf{1}}\otimes y^n \equiv \rho_n(P){\bf{1}}\otimes y^n \pmod{W^{(n+\frac{1}{2})}};$
\item $P{\bf{1}}\otimes y^n \in W^{(n+\frac{1}{2})}$ if and only if $\rho_n(P)=0$.
\end{enumerate}
\end{lemma}
\begin{proof}
$(1)$ It is enough to show it for homogeneous $P$ by induction on $\deg(P)$. If $\deg(P)=0$, i.e. $P\in\C$, then
$P{\bf{1}}\otimes y^n \equiv \rho_n(P){\bf{1}}\otimes y^n \pmod {W^{(n+\frac{1}{2})}}$
is trivial.
Suppose $ Q\in \UU(\DD^-)$ is a homogeneous element
and the result holds for this $Q$.
For $k\in\N,r\in\frac{1}{2}+\Z_+$, using Lemma \ref{main.11} we have
\begin{eqnarray*}
&d_{-k}(Q{\bf{1}}\otimes y^{n+k})
=(d_{-k}+\al-\be k+\deg(Q)-n-k)Q{\bf{1}}\otimes y^{n}\in W^{(n+\frac{1}{2})},\\
&h_{-r}(Q{\bf{1}}\otimes y^{n+r})
=(h_{-r}+\delta_{n+r-\deg(Q)\in\Z}+\ga \delta_{n+r-\deg(Q)\notin\Z})Q{\bf{1}}\otimes y^{n}\in W^{(n+\frac{1}{2})}.
\end{eqnarray*}
By equations (\ref{eq2}), these indicate that
\begin{equation*}
\begin{aligned}
d_{-k}Q{\bf{1}}\otimes y^{n}&\equiv -(\al-\be k+\deg(Q)-n-k)Q{\bf{1}}\otimes y^{n}  \\
&\equiv -(\al-\be k+\deg(Q)-n-k)\rho_n(Q){\bf{1}}\otimes y^{n}  \\
&\equiv \rho_n(d_{-k}Q){\bf{1}}\otimes y^{n} \pmod {W^{(n+\frac{1}{2})}};\\
h_{-r}Q{\bf{1}}\otimes y^{n}
&\equiv -(\delta_{n+r-\deg(Q)\in\Z}+\ga \delta_{n+r-\deg(Q)\notin\Z})Q{\bf{1}}\otimes y^{n}  \\
&\equiv -(\delta_{n+r-\deg(Q)\in\Z}+\ga \delta_{n+r-\deg(Q)\notin\Z})\rho_n(Q){\bf{1}}\otimes y^{n} \\
&\equiv \rho_n(h_{-r}Q){\bf{1}}\otimes y^{n} \pmod {W^{(n+\frac{1}{2})}}.
\end{aligned}
\end{equation*}
Then the result follows from induction on $\deg(Q)$.

$(2)$ It follows from $(1)$ and Lemma \ref{main.9}.
\end{proof}

\begin{remark}
 For the Verma module $M_{\VV}(c,h)$ over $\VV$,
it is well-known that there exist two homogeneous elements $P_1, P_2\in U(\VV^-)$
such that $U(\VV^-)P_1w_1+U(\VV^-)P_2w_1$ is the unique maximal proper submodule of $M_{\VV}(c,h)$,
where $P_1, P_2$ are allowed to  be zero and $w_1$ is the
highest weight vector in $M_{\VV}(c,h)$.
\end{remark}
\begin{lemma}\label{main.13}
Let $c,h\in\C,l\in\C^*$. There exist   two homogeneous elements $Q_1,Q_2\in\UU(\DD^-)$
such that the unique maximal proper submodule $J(c,h,l)$ of $M(c,h,l)$
is generated by $Q_1{\bf1},Q_2{\bf1}$, where $Q_1, Q_2$ are allowed to  be zero.
\end{lemma}
\begin{proof} We consider the $\DD$-module isomorphism
in Lemma \ref{main.3}
$$\pi_1:M(c,h,l)\rightarrow M_{{\VV}}(c-1,h-\frac{1}{16})^{\DD}\otimes M_{\HH}(l)^{\DD}.$$
We know that the unique maximal proper submodule $J_\VV(c,h)$ of $M_{{\VV}}(c-1,h-\frac{1}{16})$
is generated by two singular weight vectors $P_1w_1,P_2w_1$ from the above remark,
where $w_1$ is the highest weight vector of $M_{\VV}(c-1,h-\frac{1}{16})$
and $P_1, P_2\in U(\VV^-)$.
Then using Lemma \ref{main.7} we know that the unique maximal proper submodule of $M_{{\VV}}(c-1,h-\frac{1}{16})^{\DD}\otimes M_{\HH}(l)^{\DD}$ is $J_\VV(c,h)\otimes M_{\HH}(l)^{\DD}$ which is
generated by $P_1w_1\otimes w_2,P_2w_1\otimes w_2$ as a $\DD$-submodule,
where $w_2$ is the highest weight vector of $M_{\HH}(l)$. Then
$J(c,h,l)=\pi^{-1}(J_\VV(c,h)\otimes M_{\HH}(l)^{\DD})$ which is
generated by $\pi_1^{-1}(P_1w_1\otimes w_2), \pi_1^{-1}(P_2w_1\otimes w_2)$ as a $\DD$-submodule.
Thus  there exist two homogeneous elements $Q_1,Q_2\in\UU(\DD^-)$
such that $Q_1{\bf1}=\pi_1^{-1}(P_1w_1\otimes w_2),Q_2{\bf1}=\pi_1^{-1}(P_2w_1\otimes w_2)$.
So, the $J(c,h,l)$ is generated by $Q_1{\bf1},Q_2{\bf1}$.
\end{proof}

\begin{proposition}\label{main.14}
Suppose that $c,h,\al,\be\in\C,l, \ga\in\C^*$. Then $L(c,h,l)\otimes A(\al,\be,\ga)$
is irreducible if and only if
$(\rho_n(Q_1),\rho_n(Q_2))\ne (0,0)$
for all $n\in \frac{\Z}{2}$,
where $Q_1,Q_2\in\UU(\DD^-)$ are   in Lemma \ref{main.13}
\end{proposition}

\begin{proof}
Let $J=U(\DD^-)Q_1{\bf{1}}+U(\DD^-)Q_2{\bf{1}}$ be the unique maximal proper submodule of $M(c,h,l)$,
$\pi$ be the canonical map from $M(c,h,l)\otimes A(\al,\be,\ga)$ to $L(c,h,l)\otimes A(\al,\be,\ga)$.
First, we prove the following statement.
\begin{claim}\label{cla.1}
$L(c,h,l)\otimes A(\al,\be,\ga)$ is irreducible if and only if
$M(c,h,l)\otimes y^n=W^{(n+\frac{1}{2})}_n+J\otimes y^n$ for all $n\in\frac{\Z}{2}.$
\end{claim}
($\Rightarrow$) Since $W^{(n+\frac{1}{2})},J\otimes A(\al,\be,\ga)$ are submodules of $M(c,h,l)\otimes A(\al,\be,\ga)$, and $W^{(n+\frac{1}{2})}\nsubseteq J\otimes A(\al,\be,\ga)$,
 we see that
$$J\otimes A(\al,\be,\ga)\subsetneq W^{(n+\frac{1}{2})}+J\otimes A(\al,\be,\ga)\subseteq M(c,h,l)\otimes A(\al,\be,\ga).$$
But $L(c,h,l)\otimes A(\al,\be,\ga)$ is irreducible,
so $W^{(n+\frac{1}{2})}+J\otimes A(\al,\be,\ga)= M(c,h,l)\otimes A(\al,\be,\ga).$
Thus we have $M(c,h,l)\otimes y^n=W^{(n+\frac{1}{2})}_n+J\otimes y^n,\forall n\in\frac{\Z}{2}$.

($\Leftarrow$) Assume that $V_1$ is a nonzero submodule of $L(c,h,l)\otimes A(\al,\be,\ga)$,
then there exists $k\in\frac{\Z}{2}$
such that $\bar{\bf{1}}\otimes y^i\in V_1, \forall k\leq i\in\frac{\Z}{2}$  by Lemma \ref{main.10}.
It shows that $\pi({W^{(k)}+J\otimes A(\al,\be,\ga)})\subseteq V_1$. From Lemma \ref{main.11} (3) we know that
$$W^{(k)}_{k-\frac{1}{2}}\oplus \C({\bf{1}}\otimes y^{k-\frac{1}{2}})=M(c,h,l)\otimes y^{k-\frac{1}{2}}=W^{(k)}_{k-\frac{1}{2}}+J\otimes y^{k-\frac{1}{2}}\subseteq \pi^{-1}(V_1).$$
So $\bar{\bf{1}}\otimes y^{k-\frac{1}{2}}\in V_1$.
Inductively, $\bar{\bf{1}}\otimes y^{i}\in V_1$ for all $i\in\frac{\Z}{2}$.
Then it is easy to see that $V_1=L(c,h,l)\otimes A(\al,\be,\ga)$.
Therefore, $L(c,h,l)\otimes A(\al,\be,\ga)$ is irreducible.
Claim \ref{cla.1} is proved.

Using Claim 1 and   Lemmas \ref{main.9}, \ref{main.11}, \ref{main.12} we see that
\begin{equation*}
\begin{aligned} &L(c,h,l)\otimes A(\al,\be,\ga) \text{ is irreducible }\\
\iff &W^{(n+\frac{1}{2})}_n\oplus \C ({\bf 1}\otimes y^n)=M(c,h,l)\otimes y^n=W^{(n+\frac{1}{2})}_n+J\otimes y^n, \forall  n\in\frac{\Z}{2}\\
\iff &
(U(\DD^-)Q_1{\bf{1}}+U(\DD^-)Q_2{\bf{1}})\otimes y^n
\nsubseteq W^{(n+\frac{1}{2})}_n, \forall  n\in\frac{\Z}{2}  \\
\iff & \{Q_1{\bf{1}}\otimes y^n, Q_2{\bf{1}}\otimes y^n\}\nsubseteq W^{(n+\frac{1}{2})}_n , \forall  n\in\frac{\Z}{2} \\
\iff & \{\rho_n(Q_1),\rho_n(Q_2)\}\neq\{0,0\}, \forall  n\in\frac{\Z}{2}.
\end{aligned}
\end{equation*}
Now, we complete the proof of the theorem.
\end{proof}
Now we combine Proposition \ref{main.8} and Proposition \ref{main.14} into the following main theorem.
\begin{theorem}\label{main.15}
Let  $c,h,l,\al,\be\in\C,\ga\in\C^*$.
\begin{enumerate}[$(1)$]
  \item If $l=0$, then $L(c,h,l)\otimes A'(\al,\be)^{\DD}$ is an irreducible $\DD$-module
  if and only if $L_{\VV}(c,h)\otimes A'(\al,\be)$ is an irreducible $\VV$-module.
  \item If $l\ne0$, then $L(c,h,l)\otimes A'(\al,\be)^{\DD}$ is an irreducible $\DD$-module
  if and only if $L_{\VV}(c-1,h-\frac{1}{16})\otimes A'(\al,\be)$ is an irreducible $\VV$-module.
  \item If $l=0$, then $L(c,h,l)\otimes A(\al,\be,\ga)$ is irreducible.
  \item If $l\ne 0$, then $L(c,h,l)\otimes A(\al,\be,\ga)$ is irreducible if and only if
$(\rho_n(Q_1),\rho_n(Q_2))\ne (0,0)$ for all $n\in \frac{\Z}{2}$ where $Q_1,Q_2\in\UU(\DD^-)$ are   in Lemma \ref{main.13}.
\end{enumerate}
\end{theorem}
\begin{remark}\label{rem1} {\rm
Suppose $l\ne 0$ and the unique maximal submodule of $M(c,h,l)$ is generated by $Q_1{\bf1},Q_2{\bf1}$ for some $Q_1,Q_2\in\UU(\DD^-)$ as in Lemma \ref{main.13}.
For $k\in\frac{\Z}{2},$ the image of $W^{(k)}$ under the canonical map $\pi: M(c,h,l)\otimes A(\al,\be,\ga)\to L(c,h,l)\otimes A(\al,\be,\ga)$ is denoted by $\overline{W^{(k)}}$.
Then we have a sequence of submodules of $L(c,h,l)\otimes A(\al,\be,\ga):$
\begin{equation}\label{eq3}
\cdots\subseteq \overline{W^{(k+\frac{1}{2})}}\subseteq \overline{W^{(k)}}\subseteq\overline{W^{(k-\frac{1}{2})}}\subseteq\cdots
\end{equation}
From the proof of Proposition \ref{main.14}, the inclusion $\overline{W^{(k+\frac{1}{2})}}\subseteq \overline{W^{(k)}}$ is proper
if and only if \\
$(\rho_k(Q_1),\rho_k(Q_2))=(0,0)$.
Denote by $\Lambda$ the set of all $n\in\frac{\Z}{2}$ such that $(\rho_n(Q_1),\rho_n(Q_2))=(0,0)$.

If $(Q_1,Q_2)=(0,0)$, then $\Lambda=\frac{\Z}{2}$
and all inclusions $\overline{W^{(k+\frac{1}{2})}}\subseteq \overline{W^{(k)}}$
are proper for $k\in\frac{\Z}{2}$.
Moreover, $\overline{W^{(k)}}/\overline{W^{(k+\frac{1}{2})}}$ is a highest weight module
with highest weight $\al+h-k$ for $k\in\Lambda$.

If $(Q_1,Q_2)\ne(0,0)$, then $\Lambda$ is a finite subset of $\frac{\Z}{2}$.
In sequence \ref{eq3},
the inclusion $\overline{W^{(k+\frac{1}{2})}}\subseteq \overline{W^{(k)}}$ is proper
if and only if $k\in\Lambda$. Without lose of generality,
we can assume that $\Lambda=\{k_1, k_2, \cdots, k_t\}$ with $k_1<k_2<\cdots<k_t$,
then sequence \ref{eq3} can be simplified as following:
$$0\subsetneq\overline{W^{(k_{t+1})}}\subsetneq\overline{W^{(k_t)}}\subsetneq\cdots\subsetneq \overline{W^{(k_1)}}=L(c,h,l)\otimes A(\al,\be,\ga)$$
for any $k_{t+1}>k_t$.
It shows that $\overline{W^{(k_{i+1})}}=\overline{W^{(k_i+\frac{1}{2})}}$ for $1\leq i\leq t$.
However, we know that $W^{(k_i)}$ is generated by $W^{(k_i+\frac{1}{2})}$ and ${\bf1}\otimes y^{k_i}$,
thus $\overline{W^{(k_i)}}$ is generated by $\overline{W^{(k_{i+1})}}$ and $\overline{{\bf1}\otimes y^{k_i}}$.
This implies that $\overline{W^{(k_i)}}/\overline{W^{(k_{i+1})}}$ is a highest weight module
with highest weight $\al+h-k_i$.
In particular, we obtain that $\overline{W^{(k_{t+1})}}=\overline{W^{(k_t+\frac{1}{2})}}$
is an irreducible weight $\DD$-module.}
\end{remark}

\section{Isomorphism theorem}
In the section, we determine the necessary and sufficient conditions for   two of
the tensor product modules studied in Section 3 to be isomorphic.
\begin{lemma}\label{main.15'}
Let $\al,\be\in\C,\ga\in\C^*$. Then
$A(\al,\be,\ga)\cong A(\al+n,\be,\ga)$ for any $n\in\frac{\Z}{2}$.
\end{lemma}
\begin{proof}
For any $n\in \frac{\Z}{2}$, it is straightforward  to verify that  the linear map
$$F: A(\al,\be,\ga)\rightarrow A(\al+n,\be,\ga), \quad v_k\rightarrow (\delta_{n\in\Z}+\delta_{n\notin\Z}(\delta_{k\in\Z}+\ga\delta_{k\notin\Z}))v_{k+n}$$
 is an isomorphism of $\DD$-modules.
\end{proof}

\begin{proposition}\label{main.16}
Let $c,h,l,\al,\be,c_1,h_1,l_1,\al_1,\be_1,\in\C,\ga,\ga_1\in\C^*$.
Then $L(c,h,l)\otimes A(\al,\be,\ga)\cong L(c_1,h_1,l_1)\otimes A(\al_1,\be_1,\ga_1)$
if and only if $ ( \be,\ga, c,h,l)= (\be_1,\ga_1,c_1,h_1,l_1)$ and $\alpha-\alpha_1\in\frac{\Z}2$.
\end{proposition}
\begin{proof}
From  Lemma \ref{main.15'}, we may assume that $0\leq\mathfrak{Re}\al,\mathfrak{Re}\al_1<\frac{1}{2}$.
Since the  “if part” is trivial,
we only need to show the “only if part”.

We identify the $\DD$-modules $L(c,h,l)\otimes A(\al,\be,\ga)$ with $V_1=L(c,h,l)\otimes \C[y^{\pm\frac{1}{2}}]$,   and $L(c_1,h_1,l_1)\otimes A(\al_1,\be_1,\ga_1)$ with $V_2=L(c_1,h_1,l_1)\otimes\C[y^{\pm\frac{1}{2}}]$.
Let ${\overline{\bf1}}$ and ${\overline{\bf1'}}$ be the highest weight vectors of $L(c,h,l)$ and $L(c_1,h_1,l_1)$ respectively.
It is clear that $$c=c_1,l=l_1.$$

Suppose $\varphi:  V_1\to V_2$
is a $\DD$-module isomorphism.
For  $k\in\frac{\Z}{2}$, let $\varphi({\overline{\bf1}} \otimes y^k)=\sum_{i=0}^sP_i{\overline{\bf1'}}\otimes y^t$
for some $t\in\frac{\Z}{2}$ where $P_i\in U(\DD^-)_{-i}$ is homogeneous.
Note that  $\varphi({\overline{\bf1}}\otimes y^k)$ and ${\overline{\bf1}}\otimes y^k$
have the same weight. We deduce that
\begin{equation}\label{alpha-h}\al+h-k=\al_1+h_1-t.\end{equation}
For any $m,n>s$, we compute
\begin{equation}\label{eq4.3}
\begin{split}
 (\delta_{k\in\Z}&+\ga\delta_{k\notin\Z})\varphi({\overline{\bf1}}\otimes y^{k+m+n+\frac{1}{2}})=\varphi(h_{\frac{1}{2}+m+n}({\overline{\bf1}}\otimes y^k))
\\
&=
h_{\frac{1}{2}+m+n}\varphi({\overline{\bf1}}\otimes y^k)
=\sum_{i=0}^{s}(\delta_{t+i\in\Z}+\ga_1\delta_{t+i\notin\Z})P_i{\overline{\bf1'}}\otimes y^{t+m+n+\frac{1}{2}},
\end{split}
\end{equation}
\begin{equation}\label{eq4.1}
\begin{split}
 &(\delta_{k\in\Z}+\ga\delta_{k\notin\Z})(\al+\be m-k-n-\frac{1}{2})\varphi({\overline{\bf1}}\otimes y^{k+m+n+\frac{1}{2}})\\
=&\varphi(d_mh_{\frac{1}{2}+n}({\overline{\bf1}}\otimes y^k))=d_mh_{\frac{1}{2}+n}\varphi({\overline{\bf1}}\otimes y^k)
=\sum_{i=0}^{s}d_mh_{\frac{1}{2}+n}(P_i{\overline{\bf1'}}\otimes y^t)\\
=&\sum_{i=0}^{s}(\delta_{t+i\in\Z}+\ga_1\delta_{t+i\notin\Z})
(\al_1+\be_1 m-i-t-n-\frac{1}{2})P_i{\overline{\bf1'}}\otimes y^{t+m+n+\frac{1}{2}},
\end{split}
\end{equation}
Comparing the above two equations we obtain that
$$\al+\be m-k-n-\frac{1}{2}=\al_1+\be_1 m-i-t-n-\frac{1}{2}, \forall m,n>s,0\leq i\leq s$$ with $P_i{\bf 1}'\ne 0$.
Combining with (\ref{alpha-h}) we deduce that
$$\be=\be_1,\al=\al_1,i= h-h_1, t=k-h+h_1.$$
Since $i=-\deg(P_i)\ge0$, so $h- h_1\ge0$. By considering $\varphi^{-1}$ in the above arguments we deduce that $h_1- h\ge 0$.
 We must have
$$h=h_1 , t=k, s=0.$$
Therefore, for any $k\in\frac{\Z}{2}$, we may assume that
 $\varphi({\overline{\bf1}}\otimes y^k)=a_k{\overline{\bf1'}}\otimes y^k$ where $a_k\in\C^*$.
From
\begin{equation*}
\begin{split}
&\ga a_{k+m+n+1}{\overline{\bf1'}}\otimes y^{k+m+n+1}
=\ga\varphi({\overline{\bf1}}\otimes y^{k+m+n+1})=\varphi(h_{\frac{1}{2}+m}h_{\frac{1}{2}+n}({\overline{\bf1}}\otimes y^k))\\
=&h_{\frac{1}{2}+m}h_{\frac{1}{2}+n}\varphi({\overline{\bf1}}\otimes y^k)
=h_{\frac{1}{2}+m}h_{\frac{1}{2}+n}(a_k{\overline{\bf1'}}\otimes y^{k})
=a_k\ga_1{\overline{\bf1'}}\otimes y^{k+m+n+1}, \forall m,n\in\Z_+,
\end{split}
\end{equation*}
We deduce that
$$\ga a_{k+m+n+1}=a_k\ga_1, \forall m,n\in\Z_+,$$
yielding that $\gamma=\gamma_1$. We complete the proof.
\end{proof}

\begin{remark}
In Proposition \ref{main.16} we can have similar results for any highest weight modules, not only for irreducible highest weight modules.
\end{remark}
\begin{theorem}\label{main.17}
Let $c,h,l,\al,\be,c_1,h_1,l_1,\al_1,\be_1,\in\C,\ga,\ga_1\in\C^*$.
Then
\begin{enumerate}[$(1)$]
    \item $L(c,h,l)\otimes A'(\al,\be)^{\DD}\cong L(c_1,h_1,l_1)\otimes A'(\al_1,\be_1)^{\DD}$
if and only if $ ( c,h,l)= (c_1,h_1,l_1)$,  $\alpha-\alpha_1\in\Z$, and $\be=\be_1$ or $\{\be,\be_1\}=\{0,-1\}$.
    \item $L(c,h,l)\otimes A(\al,\be,\ga)\cong L(c_1,h_1,l_1)\otimes A(\al_1,\be_1,\ga_1)$
if and only if $ ( \be,\ga, c,h,l)= (\be_1,\ga_1,c_1,h_1,l_1)$ and $\alpha-\alpha_1\in\frac{\Z}2$.
\item $L(c,h,l)\otimes A(\al,\be,\ga)\not\cong L(c_1,h_1,l_1)\otimes A'(\al_1,\be_1)^\DD$.
\end{enumerate}
\end{theorem}
\begin{proof}
(1) The “if part” is trivial, we only need to show the “only if part”.
It is clear that $l=l_1$ from $L(c,h,l)\otimes A'(\al,\be)^{\DD}\cong L(c_1,h_1,l_1)\otimes A'(\al_1,\be_1)^{\DD}$.
If $l=l_1\ne 0$, using Lemmata \ref{main.3} and   \ref{main.7} we know that
\begin{eqnarray*}
&&L_{\VV}(c-1,h-\frac{1}{16})\otimes A'(\al,\be)\cong L_{\VV}(c_1-1,h_1-\frac{1}{16})\otimes A'(\al_1,\be_1),\\
&&M_{\HH}(l)\cong M_{\HH}(l_1).
\end{eqnarray*}
Thanks to \cite[Theorem 2]{CGZ},
$$L_{\VV}(c-1,h-\frac{1}{16})\cong L_{\VV}(c_1-1,h_1-\frac{1}{16}),\quad
A'(\al,\be)\cong A'(\al_1,\be_1).$$
Thus $ (c,h,l)= (c_1,h_1,l_1)$,  $\alpha-\alpha_1\in\Z$, $\be=\be_1$ or $\{\be,\be_1\}=\{0,-1\}$ by Lemma \ref{main.1}.

If $l=l_1=0$, then $L(c,h,l)\cong L_{\VV}(c,h)^{\DD}, L(c_1,h_1,l_1)\cong L_{\VV}(c_1,h_1)^{\DD}$.
Thus
$$L_{\VV}(c,h)\otimes A'(\al,\be)\cong L_{\VV}(c_1,h_1)\otimes A'(\al_1,\be_1).$$
So, $L_{\VV}(c,h)\cong L_{\VV}(c_1,h_1),A'(\al,\be)\cong A'(\al_1,\be_1)$ by Theorem 2 in \cite{CGZ}.
We also have $ ( \ga, c,h,l)= (\ga_1,c_1,h_1,l_1)$,  $\alpha-\alpha_1\in\Z$, $\be=\be_1$ or $\{\be,\be_1\}=\{0,-1\}$.

(2) This is  Proposition \ref{main.16}.

(3) From (\ref{eq1}) we know that $h_{\frac12}$ does not annihilate any weight vector in $L(c,h,l)\otimes A(\al,\be,\ga)$,  while it annihilates
${\overline{\bf 1'}}\otimes A'(\al_1,\be_1)^\DD$. So
$L(c,h,l)\otimes A(\al,\be,\ga)\not\cong L(c_1,h_1,l_1)\otimes A'(\al_1,\be_1)^\DD$.
\end{proof}

\section{Examples}

In this section, we show some concrete examples to illustrate  Theorem \ref{main.15}, and  to give irreducible weight modules with infinite dimensional weight spaces
over the mirror-twisted Heisenberg-Virasoro algebra $\DD$.

\begin{example}
We can find many irreducible Verma modules $M_{\VV}(c,h)$ over $\VV$ in \cite{A}.
For all these irreducible Verma modules, the tensor product
$M_{\VV}(c,h)^{\DD}\otimes A(\al,\be,\ga)$ is an irreducible weight $\DD$-module
with infinite dimensional weight spaces  by Theorem \ref{main.15} (3).
Thus we obtain irreducible weight modules with infinite dimensional weight spaces.
\end{example}

In the rest of this section, we present some examples for  the tensor product $\DD$-modules \\
$L(c,h,l)\otimes A(\al,\be,\ga)$ with $l\ne 0$. We will use $w_1$ to denote the highest weight vector of $M_{\VV}(c,h)$.
Let us  start with an  example where the   maximal proper submodule $J(c,h,l)$ of $M(c,h,l)$ is generated by one singular vector.

\begin{example}
Let   $l\ne 0$. Then
$M(2,\frac{1}{16},l)\cong M_{\VV}(1,0)^{\DD}\otimes M_{\HH}(l)^{\DD}$
by Lemma \ref{main.3}.
From \cite{A}, we know that the unique maximal proper submodule of $M_{\VV}(1,0)$ is generated by $d_{-1}w_1$. Using (\ref{Vertex}) we deduce that
the maximal proper submodule of $M(2,\frac{1}{16},l)$ is generated by $(d_{-1}-\frac{1}{2l}h_{-\frac{1}{2}}^2){\bf 1}$.
Therefore,
$L(2,\frac{1}{16},l)\otimes A(\al,\be,\ga)$ is irreducible
if and only if $$\rho_n(d_{-1}-\frac{1}{2l}h_{-\frac{1}{2}}^2)
=\be-\al+1+n-\frac{1}{2l}\ga\ne 0,\forall n\in\frac{\Z}{2}.$$
If $n=\al-\be-1+\frac{1}{2l}\ga\in \frac{\Z}{2}$, then $\overline{W^{(n+\frac{1}{2})}}$
is the unique minimal submodule of $L(2,\frac{1}{16},l)\otimes A(\al,\be,\ga)$.
Of course, it is the only irreducible submodule.
Moreover, $(L(2,\frac{1}{16},l)\otimes A(\al,\be,\ga))/\overline{W^{(n+\frac{1}{2})}}$ is a highest weight module
with highest weight vector ${\bf 1}\otimes y^{n}$ of weight $(2,\be-\frac{\ga}{2l}+\frac{17}{16},l)$.
By Lemma \ref{main.4}, we know that $(L(2,\frac{1}{16},l)\otimes A(\al,\be,\ga))/\overline{W^{(n+\frac{1}{2})}}$
is irreducible if  $\be-\frac{\ga}{2l}+1 \ne (\frac k2)^2$ for any $k\in\Z$.
\end{example}

\begin{example}
By Lemma \ref{main.3}, we have
$M(2,\frac{5}{16},l)\cong M_{\VV}(1,\frac{1}{4})^{\DD}\otimes M_{\HH}(l)^{\DD}$
when $l\ne 0$.
From \cite{A}, we know that the unique maximal proper submodule of $M_{\VV}(1,\frac{1}{4})$ is generated by $(d_{-1}^2-d_{-2})w_1$.
 Using (\ref{Vertex}) we deduce that the maximal proper submodule of $M(2,\frac{5}{16},l)$ is generated by
$Q{\bf 1}$ where $$Q=d_{-1}^2-d_{-2}-\frac{1}{l}d_{-1}h_{-\frac{1}{2}}^2
+\frac{1}{4l^2}h_{-\frac{1}{2}}^4+\frac{3}{2l}h_{-\frac{3}{2}}h_{-\frac{1}{2}}.$$
Therefore,
$L(2,\frac{5}{16},l)\otimes A(\al,\be,\ga)$ is irreducible
if and only if
$\rho_n(Q)\ne 0,\forall n\in\frac{\Z}{2}.$
By direct computation, we have following equations
\begin{eqnarray*}
&&\rho_n(d_{-1}^2)=(\al-\be-2-n)(\al-\be-1-n),\\
&&\rho_n(d_{-2})=-(\al-2\be-2-n),\\
&&\rho_n(d_{-1}h_{-\frac{1}{2}}^2)=-\ga(\al-\be-2-n),\\
&&\rho_n(h_{-\frac{1}{2}}^4)=\ga^2,\\
&&\rho_n(h_{-\frac{3}{2}}h_{-\frac{1}{2}})=\ga.
\end{eqnarray*}
Thus we have
\begin{eqnarray*}
&&\rho_n(Q)=\rho_n(d_{-1}^2-d_{-2}-\frac{1}{l}d_{-1}h_{-\frac{1}{2}}^2
+\frac{1}{4l^2}h_{-\frac{1}{2}}^4+\frac{3}{2l}h_{-\frac{3}{2}}h_{-\frac{1}{2}})\\
&=&(\al-\be-2-n)(\al-\be-1-n)+(\al-2\be-2-n)+\frac{\ga}{l}(\al-\be-2-n)+\frac{1}{4l^2}\ga^2+\frac{3}{2l}\ga.
\end{eqnarray*}
Solving the equation $\rho_n(Q)=0$, we get
$$n=(\al-\be-1+\frac{\ga}{2l})\pm\sqrt{\be-\frac{\ga}{2l}+1}.$$
Let $$n_1=(\al-\be-1+\frac{\ga}{2l})+\sqrt{\be-\frac{\ga}{2l}+1},\,\,\,\, n_2=(\al-\be-1+\frac{\ga}{2l})-\sqrt{\be-\frac{\ga}{2l}+1}.$$

$(1)$ We know that $L(2,\frac{5}{16},l)\otimes A(\al,\be,\ga)$ is irreducible if and only if $n_1,n_2\notin\frac{\Z}{2}$.

$(2)$ Suppose only one of  $n_1$ and $n_2$ is in $\frac{\Z}{2}$, say
$n_1\in\frac{\Z}{2}$.
Then $\overline{W^{(n_1+\frac{1}{2})}}$
is the unique irreducible submodule of $L(2,\frac{5}{16},l)\otimes A(\al,\be,\ga)$.
And the quotient module $(L(2,\frac{5}{16},l)\otimes A(\al,\be,\ga))/\overline{W^{(n_1+\frac{1}{2})}}$
is a highest weight module with highest weight
$(2,\be+\frac{21}{16}-\frac{\ga}{2l}-\sqrt{\be-\frac{\ga}{2l}+1},l)$.
By Lemma \ref{main.4}, we know that $(L(2,\frac{5}{16},l)\otimes A(\al,\be,\ga))/\overline{W^{(n_1+\frac{1}{2})}}$
is irreducible if further $\be-\frac{\ga}{2l}+\frac{5}{4}-\sqrt{\be-\frac{\ga}{2l}+1}
\ne(\frac{k}{2})^2$ for any $k\in\Z$.

$(3)$ Suppose that  $n_1,n_2\in\frac{\Z}{2}$. Without lose of generality,
we can assume that $n_1\leq n_2$,
then we have a sequence of submodules of $L(2,\frac{5}{16},l)\otimes A(\al,\be,\ga):$
$$0\subsetneq \overline{W^{(n_2+\frac{1}{2})}}\subsetneq \overline{W^{(n_2)}}\subseteq \overline{W^{(n_1)}}=L(2,\frac{5}{16},l)\otimes A(\al,\be,\ga)$$
where the inclusion $\overline{W^{(n_2)}}\subseteq \overline{W^{(n_1)}}$ is proper if and only if $n_1<n_2$.
Here, $\overline{W^{(n_2+\frac{1}{2})}}$
is the unique irreducible submodule of $L(2,\frac{5}{16},l)\otimes A(\al,\be,\ga)$
and $\overline{W^{(n_2)}}/\overline{W^{(n_2+\frac{1}{2})}}$ is a highest weight module
with highest weight $(2,\be+\frac{21}{16}-\frac{\ga}{2l}+\sqrt{\be-\frac{\ga}{2l}+1},l)$.
By Lemma \ref{main.4}, we know that $\overline{W^{(n_2)}}/\overline{W^{(n_2+\frac{1}{2})}}$ is irreducible if
$\be-\frac{\ga}{2l}+\frac{5}{4}+\sqrt{\be-\frac{\ga}{2l}+1}\ne(\frac{k}{2})^2$ for any $k\in\Z$.
As for quotient module $(L(2,\frac{5}{16},l)\otimes A(\al,\be,\ga))/\overline{W^{(n_2)}}$,
it is also a highest weight module with highest weight $(2,\be+\frac{29}{16}-\frac{\ga}{2l}+\sqrt{\be-\frac{\ga}{2l}+1},l)$ if $n_1<n_2$.
Otherwise, $L(2,\frac{5}{16},l)\otimes A(\al,\be,\ga)=\overline{W^{(n_2)}}$.

\end{example}
To conclude this section, we give an example such that
the unique maximal submodule $J(c,h,l)$ of $M(c,h,l)$ is generated by two singular vectors.
\begin{example}
Let $l\ne 0, c=h=0$. By Lemma \ref{main.3} we know that  $M(1,\frac{1}{16},l)\cong M_{\VV}(0,0)^{\DD}\otimes M_{\HH}(l)^{\DD}$.
 It is clear that the unique maximal submodule of $M_{\VV}(0,0)$
is generated by $d_{-1}w_1,d_{-2}w_1$.
 Using (\ref{Vertex}) we deduce that the maximal proper submodule of $M(1,\frac{1}{16},l)$ is generated by
$Q_1w$ and $Q_2{\bf 1}$ where $Q_1=d_{-1}-\frac{1}{2l}h_{-\frac{1}{2}}^2,Q_2=d_{-2}-\frac{1}{l}h_{-\frac{3}{2}}h_{-\frac{1}{2}}$.
We see that
$$\rho_n(Q_1)=\be-\al+1+n-\frac{1}{2l}\ga,\quad
\rho_n(Q_2)=2\be-\al+2+n-\frac{\ga}{l}.$$
Taking $\rho_n(Q_1)=\rho_n(Q_2)=0$, we have that
$n=\al\text{ and }\be+1=\frac{\ga}{2l}.$
Thus, if $n=\al\in\frac{\Z}{2}, \be+1-\frac{\ga}{2l}=0$,
then $\overline{W^{(\al+\frac{1}{2})}}$
is the unique irreducible submodule of $L(1,\frac{1}{16},l)\otimes A(\al,\be,\ga)$
and $\overline{W^{(\al)}}/\overline{W^{(\al+\frac{1}{2})}}$ is a highest weight module
with highest weight $(1,\frac{1}{16},l)$.
Otherwise, $L(1,\frac{1}{16},l)\otimes A(\al,\be,\ga)$ is irreducible.
\end{example}

\begin{remark}
For any $c,h,\al,\be\in\C,l,\ga\in\C^*$, it is generally hard to determine the structure of nontrivial quotient modules $\overline{W^{(k)}}/\overline{W^{(k+1)}}$ for the $\DD$-module $L(c,h,l)\otimes A(\al,\be,\ga)$.
We also have the similar questions   for the Virasoro algebra and Heisenberg-Virasoro algebra, see \cite{CGZ, LZ2}.
\end{remark}

{\bf Acknowledgments:}
The research in this paper was carried out during the visit of the first
author at Wilfrid Laurier University in 2019-2021.
The first author is partially supported by CSC of China (Grant 201906340096), and NSF of China (Grants 11771410 and 11931009).
The second author is partially supported by   NSFC (11871190) and NSERC (311907-2015).
The first author wants to thank Y. Ma for helpful suggestions
when he was preparing the paper.

\end{document}